\documentclass[letterpaper, 10 pt, conference]{ieeeconf}  

\IEEEoverridecommandlockouts                              
\title{\LARGE \bf  High-Probability Convergence of SGD via Batched Updates}
\author{Feng Zhu, Robert W. Heath Jr., and Aritra Mitra
\thanks{F. Zhu and A. Mitra are with the Dept. of Electrical and Computer Engineering, North Carolina State University. Email: {\tt \{fzhu5, amitra2\}@ncsu.edu}. Robert W. Heath Jr. is with the Dept. of Electrical and Computer Engineering at the University of California, San Diego, USA. Email: {\tt rwheathjr@ucsd.edu}. This material is based upon work supported in part by the NSF under Grant No. NSF-CCF-2225555.}}
\usepackage{amsmath,amsfonts,amssymb,color,amsthm}
\usepackage{nccmath}
\usepackage{epsfig}
\usepackage{psfrag}
\usepackage{algorithm}
\usepackage{algpseudocode}
\usepackage{array}
\usepackage{epstopdf}
\usepackage{cite}
\usepackage{footmisc}
\usepackage{mathtools}
\usepackage{bbm}
\usepackage{bm}
\usepackage{comment}
\usepackage{dsfont}

\usepackage{mathtools}
\usepackage{epstopdf}
\usepackage{tikz}
\usetikzlibrary{automata, positioning, arrows.meta}
\usepackage{relsize}
\usetikzlibrary{shapes,arrows}
\usepackage{cite}
\usepackage{epstopdf}

\definecolor{winered}{rgb}{0.5,0,0}
\usepackage{scalerel}
\usepackage{xcolor}
\usepackage[colorlinks]{hyperref}

\newcommand{\mc}{\mathcal}

\newcommand{\E}{\mathbb{E}}

\renewcommand{\P}{\mathbb{P}}

\newcommand{\R}{\mathbb{R}}

\newcommand{\cF}{\mathcal{F}}

\newcommand{\cO}{\mathcal{O}}

\newcommand{\iprod}[2]{\left\langle #1,\, #2\right\rangle}
\newcommand{\bracket}[1]{\left[ #1\right]}
\newcommand{\norm}[1]{\left\| #1\right\|_2}
\newcommand{\bigo}[1]{\cO\left( #1\right)}
\newcommand{\paren}[1]{\left( #1\right)}

\AtBeginDocument{%
  \hypersetup{
    citecolor=blue,
    linkcolor=winered}}

\DeclareMathOperator*{\argmin}{\arg\!\min}

\newtheorem{problem}{Problem}

\newtheorem{theorem}{Theorem}

\newtheorem{lemma}{Lemma}
\newtheorem{remark}{Remark}

\newtheorem{assumption}{Assumption}

\begin{document}

\maketitle
\thispagestyle{empty}
\pagestyle{empty}

\begin{abstract}
Stochastic gradient descent (SGD) is the primary workhorse for large-scale optimization. While the average behavior of its iterates—typically characterized by mean-squared error bounds—is well-understood, obtaining high-probability guarantees for the last iterate remains challenging. Prior approaches to this problem have either imposed restrictive assumptions (such as bounded domains or gradients) or relied on complex proofs involving auxiliary sequences. In this work, we propose \texttt{Batched SGD}, a simple variant that partitions online samples into epochs and performs a single update per epoch using a refined, low-variance gradient estimate. Our main contribution demonstrates that this batching mechanism enables a surprisingly simple high-probability analysis that avoids both restrictive assumptions and auxiliary sequences. Under standard smoothness and norm-sub-Gaussian noise assumptions, we establish near-optimal rates for both strongly convex and non-convex objectives. Furthermore, we show that our batching idea extends naturally to federated learning (FL). We provide the first high-probability guarantees for FL, achieving logarithmic communication complexity, linear speedup in the number of agents, and resilience to data heterogeneity.
\end{abstract}

\section{Introduction}\label{sec:intro}
We consider the problem of minimizing a smooth objective function $f$. In modern large-scale machine learning (ML) tasks, either the dataset is massive or data arrives sequentially, making the computation of the full gradient $\nabla f(\cdot)$ prohibitively expensive or even infeasible. As a result, standard gradient descent (GD) becomes inaccessible in such settings. To address this challenge, the seminal work of~\cite{robbins1951stochastic} introduces the well-known stochastic gradient descent (SGD) method, which relies on inexpensive \textit{stochastic} gradient estimates in place of the full gradient. Due to its computational efficiency, scalability, and favorable convergence properties, SGD has become the workhorse of modern optimization.

Given its broad applicability, the convergence properties of SGD have been extensively studied in the literature, with a large body of work focusing on \textit{in-expectation} analyses~\cite{bottou2018optimization,moulines2011non}, which provide mean-squared error (MSE) bounds. While such guarantees characterize the average behavior, they can be less informative in modern large-scale applications, where only a limited number of runs (or even a single run) of the algorithm can be performed due to computational constraints~\cite{harvey2019tight}. This has led to an increasing interest in \textit{high-probability} analyses, which capture tail behavior and quantify the uncertainty of rare but significant deviations. 


Compared to in-expectation analyses, obtaining such high-probability guarantees is substantially more challenging (see Section~\ref{subsec:challenges}), and typically requires stronger assumptions on the noise distribution, such as light-tailedness~\cite{Orabona,liu2023high,liu2023revisiting}, captured via norm-sub-Gaussianity, to leverage suitable concentration inequalities. Nevertheless, even in the simplest setting---smooth and strongly-convex objectives with norm-sub-Gaussian noise---a significant gap exists in the understanding of high-probability guarantees for standard SGD.

Majority of the results rely on restrictive assumptions such as bounded domains~\cite{rakhlin2011making,jain2021making} or bounded noise/stochastic gradients~\cite{harvey2019tight,li2019convergence,kavis2022high}. These assumptions significantly limit the applicability of the analysis: bounded domains exclude unconstrained problems, while bounded gradient assumptions are violated even for simple quadratic objectives.

Only a limited number of works avoid such restrictive assumptions; however, they typically provide guarantees only for the \textit{averaged iterate} and rely on complex analyses~\cite{li2019convergence,liu2023high}. While iterate averaging is useful for analysis, it does not reflect the behavior of the algorithm in practice, where the \textit{last iterate} is often used. Establishing high-probability guarantees for the last iterate under general conditions remains technically challenging and is still not well understood. A notable exception is~\cite{liu2023revisiting}, which does provide a high-probability guarantee for the last iterate; however, the analysis is again technically involved and relies on an auxiliary sequence. Due to the complexity of existing analyses, it remains unclear whether they can be readily extended to more complex settings, such as those involving multiple agents in popular paradigms such as federated learning (FL)~\cite{mcmahan, konevcny}. Given this context, we now summarize our main contributions.

$\bullet$ \textbf{Batched SGD variant.} Our first contribution is to propose a variant of online SGD called \texttt{Batched SGD} that leads to surprisingly simple high-probability convergence proofs relative to prior work. Unlike standard SGD that makes updates at every time-step as soon as a new data sample arrives, the essence of our approach (detailed in Section~\ref{sec:bsgd}) lies \emph{in making fewer updates, but using more precise variance-reduced} update directions. Such directions are constructed by partitioning the online data stream into batches of equal size, and averaging the samples within each batch. The key design step is to choose the batch-size as a function of the properties of the objective function $f$. 

$\bullet$ \textbf{Last Iterate Rates with Simple Analysis.} Given $T$ samples, and a prescribed failure probability $\delta \in (0,1)$, we establish rates of $\tilde{\mc{O}}\left(\log(dT/\delta)/T\right)$ and $\tilde{\mc{O}}\left(\log(dT/\delta)/\sqrt{T}\right)$ for smooth, strongly convex (Theorem~\ref{thm:sc})
and non-convex functions (Theorem~\ref{thm:nc}), respectively, that hold with probability $1-\delta$; here, $d$ is the dimension of the underlying parameter. Our rate for the strongly convex case is directly for the last iterate. These rates match their in-expectation counterparts in terms of the dependence on $T$, and achieve a logarithmic dependence on $1/\delta$, implying tight concentration intervals. Notably, the batching strategy leads to significantly simpler proofs compared to prior work, and requires neither restrictive assumptions nor auxiliary sequences. Instead, we only need to leverage a standard variance reduction property of norm-sub-Gaussian random vectors~\cite{jin2019short}. 

$\bullet$ \textbf{Extension to Federated Learning.} In Section~\ref{sec:FL}, we show that our batching idea extends naturally to federated learning (FL), where local computations at each agent can be interpreted as forming batches between communication rounds. In this view, the batch length directly determines the communication frequency between a central server and the agents. For smooth, strongly convex losses, we show in Theorem~\ref{thm:fl_sc} that this approach yields high-probability guarantees with \textit{logarithmic} communication rounds, \textit{linear speedup} in the number of agents, and \textit{no bias due to heterogeneity} in agents' local loss functions. To our knowledge, this is the first high-probability result in FL. 

It should be noted that while mini-batching is a classical idea, identifying that batching can lead to near-optimal high-probability rates with simple proofs is the key new insight offered by our work.

\section{Problem Formulation and Motivation}\label{sec:prob_form}
We start by providing relevant technical background. Consider the following optimization problem:
\begin{equation}
    \min_{x\in\R^d} f(x),
\end{equation}
where $f:\R^d\to\R$ is the objective function. Our goal is to obtain the optimum $x^*\in\R^d$ that minimizes $f$. To achieve this goal, the classical SGD algorithm~\cite{robbins1951stochastic} operates iteratively, where new samples arrive online at each time-step. Specifically, at time-step (iteration) $t$, a query is made to an oracle, and the oracle returns a stochastic gradient $g(x_t, \xi_t)$. Here, $x_t\in\R^d$ is the parameter iterate at time-step $t$, and $\{\xi_t\}_{t\ge 0}$ is a sequence of I.I.D. (independent and identically distributed) random variables (RVs), capturing the randomness in the algorithm. The iterate $x_t$ is updated as:
\begin{equation}
    x_{t+1} = x_t - \alpha g(x_t, \xi_t),\quad\forall t=0,\cdots, T-1,\label{eqn:SGD}
\end{equation}
where $T$ denotes the horizon of the algorithm, and $\alpha$ is the step-size. It is well known that for smooth and strongly convex functions $f$, the error $f(x_T) - f(x^*)$ decays as $\bigo{1/T}$ \emph{in expectation}, when SGD in~\eqref{eqn:SGD} is run for $T$ iterations with a suitable step-size~\cite{nemirovski2009robust}. Our interest instead lies in characterizing high-probability \emph{concentration} bounds on $f(x_T) - f(x^*)$, where $\{x_t\}$ is either generated by SGD or some variant thereof. We summarize this below. 

\begin{problem}
Given a failure probability $\delta \in (0,1)$, and $T$ gradient queries (samples), compute an estimate $x_T$ of $x^*$, and quantify an error bound of the form $f(x_T)- f(x^*) \leq e(T, \delta)$, which holds with probability at least $1-\delta.$
\end{problem}

One can always translate a known in-expectation bound on $f(x_T) - f(x^*)$ to a probabilistic deviation bound using either Markov's or Chebyshev's inequality. However, the resulting error bound $e(T, \delta)$ will then have a polynomial dependence on $1/\delta$, which can be vacuously large if one seeks guarantees that fail with very low probability (e.g., say $\delta =10^{-6}$). In contrast, we aim to obtain error bounds $e(T, \delta)$ that scale logarithmically in $1/\delta$. To set the stage for such an analysis, we make certain standard assumptions~\cite{nemirovski2009robust, bottou2018optimization, jain2021making}.

\begin{assumption}[Smoothness]\label{ass:smoothness}
    The objective function $f$ is $L$-smooth, i.e., for any $x,y\in\R^d$, the following holds:
\begin{equation}
    \norm{\nabla f(x)-\nabla f(y)}\leq L\norm{x-y},
\end{equation}
where $\nabla(\cdot)$ denotes the (full) gradient operator. 
\end{assumption}

\begin{assumption}[Strong Convexity]\label{ass:SC}
The objective function $f$ is $\mu$-strongly convex, i.e., for any $x,y\in\R^d$, we have
\begin{equation}
    f(y)\geq f(x)+\iprod{\nabla f(x)}{y-x} + \frac{\mu}{2}\norm{x-y}^2.
\end{equation}
\end{assumption}

\begin{assumption}[Unbiased Gradient]\label{ass:unbiased}
The stochastic gradient $g(x,\xi)$ is unbiased, i.e., 
\begin{equation}
    \E_\xi\bracket{g(x,\xi)}=\nabla f(x), \forall x \in \mathbb{R}^d. \label{eqn:unbiased}
\end{equation}
\end{assumption}

\begin{assumption}[Norm-Sub-Gaussian Noise]\label{ass:subG}
The norm of the noise in the stochastic gradient $g(x,\xi)$ is sub-Gaussian, i.e., for any given $x\in\R^d$, the following holds:
\begin{equation}
    \E_\xi\bracket{\exp\paren{\norm{g(x,\xi)-\nabla f(x)}^2/\sigma^2}}\leq \exp\paren{1},\label{eqn:subG}
\end{equation}
where $\sigma^2>0$  is some constant known as the variance proxy.
\end{assumption}
Assumption~\ref{ass:subG} has been extensively used in the literature to develop high-probability guarantees~\cite{harvey2019tight,Orabona,liu2023revisiting}. Essentially, it implies that the tail of the noise norm distribution decays at least as fast as that of a Gaussian distribution.

The remainder of this section is dedicated to highlighting the specific challenges that arise when one aims to derive high-probability bounds for the standard SGD algorithm in~\eqref{eqn:SGD}. These challenges will directly motivate the batching strategy in Section~\ref{sec:bsgd} that leads to a short, simple, and intuitive convergence analysis. To proceed, define $r_t:=f(x_t)-f(x^*)$, and $e_t:=g(x_t,\xi_t) - \nabla f(x_t)$ as the noise in the stochastic gradient at time-step $t$. We then have the following lemma that captures the one-step descent in $r_t$.
\begin{lemma}[One-Step Descent]\label{lem:onestep}
The following holds for any $t\ge0$ by choosing $\alpha L\le 1/2$:
\begin{equation}
\medmath{
    r_{t+1}\leq r_t - \frac{\alpha}{2}\norm{\nabla f(x_t)}^2-\alpha\iprod{\nabla f(x_t)}{e_t}+L\alpha^2\norm{e_t}^2.}\label{eqn:onestep}
\end{equation}
\end{lemma}
\begin{proof}
Since the objective function $f$ is $L$-smooth, we have
\begin{equation}
\label{eqn:OG_onestep}
\medmath{
\begin{aligned}
    r_{t+1}&\leq r_t+\iprod{\nabla f(x_t)}{x_{t+1}-x_t}+\frac{L}{2}\norm{x_{t+1}-x_t}^2\\
    &\overset{(a)}{=}r_t-\alpha \iprod{\nabla f(x_t)}{g(x_t,\xi_t)}+\frac{L\alpha^2}{2}\norm{g(x_t,\xi_t)}^2\\
    &\overset{(b)}{\leq} r_t-\alpha\norm{\nabla f(x_t)}^2-\alpha\iprod{\nabla f(x_t)}{e_t}\\
    &\ +L\alpha^2\norm{\nabla f(x_t)}^2+L\alpha^2\norm{e_t}^2\\
    &\overset{(c)}{\leq} r_t - \frac{\alpha}{2}\norm{\nabla f(x_t)}^2-\alpha\iprod{\nabla f(x_t)}{e_t}+L\alpha^2\norm{e_t}^2,
\end{aligned}}
\end{equation}
where $(a)$ follows from the update rule in~\eqref{eqn:SGD}, $(b)$ holds due to the definition of $e_t$ and Jensen's inequality, and $(c)$ is the result of choosing $\alpha L\leq 1/2$. The proof is complete.
\end{proof}

From the right-hand-side (R.H.S.) of~\eqref{eqn:onestep}, we note that the one-step progress made by SGD depends on three terms: (i) the ``good term'' $-(\alpha/2) \norm{\nabla f(x_t)}^2$, (ii) the cross term $-\alpha\iprod{\nabla f(x_t)}{e_t}$, and (iii) the squared-norm term $L\alpha^2\norm{e_t}^2$. Under strong convexity, the good term can be related to the function gap $r_t$, inducing contraction toward the optimum. The main challenge is to control (ii) \& (iii) so that \textit{their combined effect does not offset this descent and can instead be subsumed into the dominant good term.}
We first focus on the relatively simpler task of controlling  $\norm{e_t}^2$.

\begin{lemma}[Bounding Noise Magnitude]\label{lem:sn} Given any $\delta\in(0, 1)$, with probability (w.p.) at least $1-\delta$, we have:
\begin{equation}
    \norm{e_t}^2\leq \sigma^2\log \left(\frac{T\exp(1)}{\delta} \right),\quad \forall t\ge 0\label{eqn:sqnm}
\end{equation}
where $\sigma^2$ is as defined in Assumption~\ref{ass:subG}.
\end{lemma}
\begin{proof}
Denote $\cF_t:=\sigma\paren{\xi_0,\xi_1,\cdots, \xi_t}$ as the sigma-algebra capturing all randomnesses up to time-step $t$ ($\cF_{-1}=\varnothing$ when $t=0$), $\P_t(\cdot):=\P(\cdot\mid\cF_{t-1})$, and $\E_t[\cdot]:=\E[\cdot\mid\cF_{t-1}]$ as the probability and expectation operators conditioned on $\cF_{t-1}$. Then, for any fixed $t\geq 0$ and $s>0$, we have:
\begin{equation}
\medmath{
\begin{aligned}
    \P_t\paren{\norm{e_t}^2\geq s}&=\P_t\paren{\exp\paren{\frac{\norm{e_t}^2}{\sigma^2}}\geq\exp\paren{\frac{s}{\sigma^2}}}\\
    &\overset{(a)}{\leq} \exp\paren{-\frac{s}{\sigma^2}}\E_t\bracket{\exp\paren{\frac{\norm{e_t}^2}{\sigma^2}}}\\
    &\overset{(b)}{\leq} \exp\paren{-\frac{s}{\sigma^2}+1},
\end{aligned}
}
\end{equation}
where $(a)$ holds due to Markov's inequality, and $(b)$ uses Assumption~\ref{ass:subG} and the fact that $x_t$ is deterministic conditioned on $\cF_{t-1}$. It then readily follows from the tower property that
\begin{equation}
\medmath{
    \P\paren{\norm{e_t}^2\geq s}=\E\bracket{\P_t\paren{\norm{e_t}^2\geq s}}\leq \exp\paren{-\frac{s}{\sigma^2}+1}.}\label{eqn:subg_proof}
\end{equation}
Setting the R.H.S. of~\eqref{eqn:subg_proof} to $\delta$ yields $s = \sigma^2\log\frac{\exp(1)}{\delta}$. The proof follows by union-bounding over all $T$ iterations. 
\end{proof}
Thanks to the light-tailedness of the gradient noise norm $\norm{e_t}$, Lemma~\ref{lem:sn} tells us that $\alpha ^2 L \Vert e_t \Vert^2_2$ is on the order of $\bigo{\alpha^2 L\sigma^2\log(T/\delta)}$. We now proceed to the most delicate part of obtaining the desired high-probability guarantee. 

\subsection{Challenges in the Analysis}
\label{subsec:challenges}
We are left to tackle the cross term $-\alpha\iprod{\nabla f(x_t)}{e_t}$, which poses the main technical difficulty of the analysis. With an in-expectation analysis, this term does not pose a challenge, as it vanishes after taking expectations conditioned on $\cF_{t-1}$ under the unbiased gradient assumption. In contrast, a high-probability analysis precludes taking expectations, and thus the cross term cannot be eliminated in this manner. 

$\bullet$ \textbf{Why Do Naive Approaches Fail?} A natural approach to control the cross term is to use Young's inequality as follows: 
\begin{equation}
    -\alpha\iprod{\nabla f(x_t)}{e_t}\leq \frac{\alpha}{4}\norm{\nabla f(x_t)}^2 + \alpha\norm{e_t}^2,\label{eqn:young_cross}
\end{equation}
where the constants are chosen so that the effect of the good term $-(\alpha/2) \norm{\nabla f(x_t)}^2$ is not completely canceled. The issue with this approach is that the $\Vert e_t \Vert^2_2$ term is now dominated by $\alpha\norm{e_t}^2$, which scales \textbf{linearly} with the step-size $\alpha$, rather than \textbf{quadratically} as $L\alpha^2\norm{e_t}^2$ in the original one-step descent analysis in~\eqref{eqn:OG_onestep}. 

Consequently, invoking the gradient domination property of strong convexity and unrolling the resulting recursion for $T$ time-steps yields a bias of order $\bigo{(\sigma^2/\mu)\log(T/\delta)}$ in the final bound for $r_T$ using Lemma~\ref{lem:sn}. Crucially, this bias term does not get hit by the step-size $\alpha$. With an in-expectation analysis, the corresponding bias scales as $\bigo{\alpha L\sigma^2/\mu}$, allowing the step-size $\alpha$ to be tuned to control the final error floor. However, the naive high-probability analysis using Young's inequality causes the bias to be \emph{step-size independent}, leading to vacuous bounds. 

$\bullet$ \textbf{Martingale-Based Methods.} An alternative approach is to observe that the cross term sequence $\{\iprod{\nabla f(x_t)}{e_t}\}_{t\geq 0}$ forms a martingale difference sequence (MDS) since $e_t$ is zero-mean conditioned on $\cF_{t-1}$. A natural attempt is then to apply Azuma-Hoeffding to control the cumulative deviation of the cross terms. The obstacle here is that the martingale increments are \textit{iterate-dependent}: $\iprod{\nabla f(x_t)}{e_t}$ depends on  ${\nabla f(x_t)}$, which cannot be uniformly bounded as required by standard martingale concentration inequalities. 

A large body of work has bypassed this difficulty by imposing restrictive assumptions of boundedness on either the stochastic gradients or on the domain~\cite{harvey2019simple,li2019convergence,harvey2019tight,rakhlin2011making,jain2021making,kavis2022high}. However, such assumptions are often unrealistic in practice and can even fail in simple settings; for instance, even for quadratic unconstrained problems, the gradients are unbounded. The few works that avoid such assumptions, e.g.,~\cite{Orabona,liu2023high}, typically establish guarantees only for the averaged iterate. A notable exception is~\cite{liu2023revisiting}, which provides guarantees for the last iterate, but relies on a significantly more involved analysis based on an auxiliary sequence.

Motivated by these issues, we ask: \textit{Can we develop a simple high-probability analysis for SGD?} We provide an affirmative answer in the next section by introducing a simple variant of SGD that offers several advantages: (i) a significantly simpler analysis; (ii) a direct guarantee on the last iterate; and (iii) a straightforward extension to the FL setting. To the best of our knowledge, no prior approach enjoys all three aforementioned benefits.

\section{Batched SGD Algorithm}\label{sec:bsgd}
Recall that in the standard SGD algorithm~\eqref{eqn:SGD}, the parameter $x_t$ is updated using $g(x_t,\xi_t)$ at each time-step $t$, as soon as a new query is made. Since $g(x_t,\xi_t)$ is constructed from a \textbf{single} sample, it suffers from high variance. Given this observation, our main idea is to develop a variant of online SGD called \texttt{Batched SGD} that makes fewer parameter updates, but each update is made using a \emph{less noisier, more precise direction that enjoys a variance-reduction property.} As we shall see, this intuitive idea will lead to simple convergence proofs. We now provide the details. 

\textbf{Core Idea.} Recalling that the horizon of the algorithm is $T$, we propose to split the $T$ samples into $K$ batches with $H$ samples each (we assume for simplicity that $T$ is divisible by both $K$ and $H$). The parameters $K, H$ will be specified in the next section. For each batch $k=0, 1, \ldots, K-1$, the $H$ samples within it are used to construct a refined gradient direction which is used to make \textbf{one single update} to the parameter at the end of the batch. 

Specifically, denote the parameter at the beginning of batch $k$ as $x^{(k)}$. During each batch $k$, $H$ stochastic gradients $\{g(x^{(k)}, \xi_{kH+\ell})\}_{\ell=0}^{H-1}$ are collected, while the corresponding iterate $x^{(k)}$ is kept ``\textit{frozen}'' from the beginning of batch $k$. Then, a \textit{refined stochastic gradient} $g_k$ is constructed as:
\begin{equation}
    g_k:=\frac{1}{H}\sum_{\ell=0}^{H-1}g(x^{(k)}, \xi_{kH+\ell}),\label{eqn:refined_gradient}
\end{equation}
where $\xi_{kH+\ell}$ is essentially the $\ell$-th sample in the $k$-th batch. Since these random variables are I.I.D.,  Lemma~\ref{lem:VR_subG} shows that $g_k$ enjoys an $H$-fold lower variance compared to the SGD gradient constructed from a single sample. Using $g_k$, the parameter $x^{(k)}$ is then updated at the end of batch $k$ as:
\begin{equation}
    x^{(k+1)}=x^{(k)}-\alpha g_k.\label{eqn:bsgd_update}
\end{equation}
The above steps are summarized in Algorithm~\ref{algo:bsgd}, and illustrated in Fig.~\ref{fig:sgd}. In the following section, we will establish that the \texttt{Batched SGD} algorithm yields a near-optimal high-probability guarantee with a much simplified proof.

\begin{figure}
    \centering
    \includegraphics[scale = 0.525]{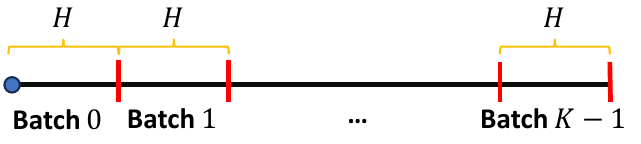}
    \caption{Illustration of \texttt{Batched SGD}, where $T$ samples are divided into $K$ batches with $H$ samples each. Red bars indicate where the updates happen.}
    \label{fig:sgd}
\end{figure}

\begin{algorithm}[t!]
\caption{\texttt{Batched SGD}} 
\label{algo:bsgd}
\begin{algorithmic}[1]
\State \textbf{Input:} Step-size $\alpha$, initial parameter $x^{(0)}$, number of samples $T$, number of batches $K$, batch-size $H$.
\For {$k=0,\ldots,K-1$} 
\For {$\ell=0,\ldots, H-1$}
\State Collect $g(x^{(k)}, \xi_{kH+\ell})$
\EndFor
\State Construct refined gradient $g_k$ as per~\eqref{eqn:refined_gradient}
\State Update parameter via~\eqref{eqn:bsgd_update}
\EndFor
\end{algorithmic}
\end{algorithm}
\section{Main Results And Analysis}\label{sec:analysis}
In this section, we present the main results of our \texttt{Batched SGD} algorithm along with the corresponding analysis. We begin with a key lemma that underpins the foundation of our analysis, followed by several insights that highlight the role of the batching idea.

\begin{lemma}[Key Technical Lemma]\label{lem:onestep_sc}
Define $r^{(k)}:=f(x^{(k)})-f(x^*)$. Suppose Assumptions~\ref{ass:smoothness} to~\ref{ass:subG} hold. Then, given any $\delta\in(0,1)$, there exists a universal constant $c$ such that with $\alpha\leq1/(2L)$, the following holds for \texttt{Batched SGD} w.p. at least $1-\delta$: 
\begin{equation}
\label{eqn:main_decomp}
    r^{(K)} \leq \underbrace{\paren{1-\frac{\alpha\mu}{2}}^Kr^{(0)}}_{T_1}+\underbrace{\frac{4c^2\sigma^2}{\mu H}\log\frac{2dT}{\delta}}_{T_2}.
\end{equation}
\end{lemma}
The proof is deferred to later in this section.
\begin{remark}
Lemma~\ref{lem:onestep_sc} provides an upper-bound on the function sub-optimality error after $K$ batches (equivalently $T$ samples). The upper-bound in~\eqref{eqn:main_decomp} comprises an exponentially decaying optimization term $T_1$, and a statistical term $T_2$ capturing the effect of noise. A key feature of our approach, enabled by the variance-reduction effect of batching, is that the noise term $T_2$ gets scaled down by the batch size $H$, which is a tunable parameter. When $H=1$, we recover the standard SGD algorithm, where the noise term $T_2$ can no longer be controlled, leading to vacuous bounds, as we had discussed earlier in Section~\ref{subsec:challenges}. 

To see how the additional degree of freedom provided by $H$ can be exploited, observe that $T_1$ decays exponentially in $K$, while the noise term $T_2$ decays linearly with $H$. As such, to balance these two terms, a $K$ on the order of $\mc{O}(\log(T))$ suffices. Since $T=KH$, this implies that the batch size $H$ is essentially $T$ up to log factors.
\end{remark}

The above intuition is formalized in the following result which specifies the choices of $K$ and $\alpha$. 

\begin{theorem}[Strongly Convex Objective, \texttt{Batched SGD}]\label{thm:sc}
Suppose the conditions in Lemma~\ref{lem:onestep_sc} hold. Then, given any $\delta \in (0,1)$, 
by choosing $K=4\kappa\log T$ and $\alpha=1/(2L)$, the following holds for \texttt{Batched SGD} w.p. at least $1-\delta$:
\begin{equation}
\label{eqn:final_sc}
    r^{(K)}\leq \frac{r^{(0)}}{T} + \frac{16c^2\kappa\sigma^2}{\mu T}\log T\log\frac{2dT}{\delta}, 
\end{equation}
where $\kappa:=L/\mu$ is the condition number.
\end{theorem}
The proof of Theorem~\ref{thm:sc} follows immediately from Lemma~\ref{lem:onestep_sc} by selecting $\alpha=1/(2L)$ and $K=4\kappa\log T$. A few takeaways are in order for Theorem~\ref{thm:sc}.

$\bullet$ \textit{Near-Optimal High-Probability Rates.} Theorem~\ref{thm:sc} reveals that a near-optimal high-probability guarantee of order $\mc{O}\left(\log(T)\log(dT/\delta)/T\right)$ is achievable using our \texttt{Batched SGD} algorithm. This rate matches the result in~\cite{liu2023revisiting} up to a $\log d$ factor, but with a significantly simpler proof.

$\bullet$ \textit{Guarantee on The Last Iterate.} Theorem~\ref{thm:sc} establishes a high-probability guarantee on the last iterate. As far as we know, the only work that achieves this last-iterate high-probability bound for SGD is~\cite{liu2023revisiting}, but with a much more involved analysis that exploits an auxiliary sequence.

$\bullet$ \textit{Logarithmic Update Rounds.} A notable feature that emerges from our analysis of \texttt{Batched SGD} is that, to achieve near-optimal high-probability bounds, it suffices to update the parameter only logarithmically many times, quantified by $K=\bigo{\log T}$. This key insight will directly inform the choice of the communication frequency in Section~\ref{sec:FL}, when we study federated learning. 

We now proceed to prove  Lemma~\ref{lem:onestep_sc}.

\noindent \textbf{Proof of Lemma~\ref{lem:onestep_sc}.} We first state a variant of~\cite[Corollary 7]{jin2019short}, which pertains to variance-reduction of norm-sub-Gaussian random vectors.

\begin{lemma}[Variance-Reduced Norm-Sub-Gaussian]\label{lem:VR_subG}
Let random vectors $X_1,\cdots,X_n\in\R^d$ be independent, zero-mean, and norm-sub-Gaussian with variance proxy $\sigma^2>0$, i.e., $\E\bracket{X_i}=0$, $\E\bracket{\exp\paren{\norm{X_i}^2/\sigma^2}}\leq \exp(1)$, $\forall i\in[n]$. Then, there exists a universal constant $c$ such that for any $\delta\in(0,1)$, with probability at least $1-\delta$:
\begin{equation}
    \norm{\frac{1}{n}\sum_{i=1}^nX_i}\leq c\sqrt{\frac{\sigma^2}{n}\log\frac{2d}{\delta}}.\label{eqn:VR_subG}
\end{equation}
\end{lemma}
The proof can be found in~\cite{jin2019short}. Lemma~\ref{lem:VR_subG} explicitly shows how averaging $n$ norm-sub-Gaussian random vectors leads to an $n$-fold reduction in variance, causing the norm of the average to concentrate more tightly around zero as $n$ increases. Using this tool, in the next lemma, we show that the gradient $g_k$ in~\eqref{eqn:refined_gradient} for the \texttt{Batched SGD} algorithm enjoys a similar variance-reduction effect. 

\begin{lemma}[Variance-Reduced Gradients]\label{lem:VR_sn}
Define $n_k:=g_k-\nabla f(x^{(k)})$, where $g_k$ is as in~\eqref{eqn:refined_gradient}. Given any $\delta\in(0, 1)$, w.p. at least $1-\delta$, the following holds (for \texttt{Batched SGD}) simultaneously for all $k=0, 1, \ldots, K-1$:
\begin{equation}
\label{eqn:VR_noise}
    \norm{n_k}^2\leq \frac{c^2\sigma^2}{H}\log\frac{2dT}{\delta},
\end{equation}
where $c$ is some universal constant.
\end{lemma}
\begin{proof}
First, fix a batch index $k$. As per the definition of $g_k$ in~\eqref{eqn:refined_gradient}, we can decompose the noise $n_k$ as
\begin{equation}
    n_k=\frac{1}{H}\sum_{\ell=0}^{H-1}n_{k,\ell},
\end{equation}
where each component $n_{k,\ell}$ is defined as
\begin{equation}
    n_{k,\ell} := g(x^{(k)}, \xi_{kH+\ell})-\nabla f(x^{(k)}).
\end{equation}
We then argue that $\{n_{k,\ell}\}_{\ell=0}^{H-1}$ satisfies the conditions in Lemma~\ref{lem:VR_subG}. To that end, first condition on $\cF_{kH-1}$ such that $x^{(k)}$ becomes deterministic, and the sole source of randomness in $\{n_{k,\ell}\}_{\ell=0}^{H-1}$ comes from $\{\xi_{kH+\ell}\}_{\ell=0}^{H-1}$, which is an \textbf{independent} sequence. In the special case of $k=0$, we do not need to condition since $x^{(0)}$ is already deterministic.

According to Assumption~\ref{ass:unbiased}, the stochastic gradient $g(x^{(k)}, \xi_{kH+\ell})$ is an unbiased estimate of the true gradient $\nabla f(x^{(k)})$, implying that the noise component $n_{k,\ell}$ is zero-mean conditioned on $\cF_{kH-1}$. Moreover, Assumption~\ref{ass:subG} states that the noise in each stochastic gradient is norm-sub-Gaussian with variance proxy $\sigma^2$, i.e.,
\begin{equation} \E_{kH}\bracket{\exp\paren{\frac{\norm{n_{k,\ell}}^2}{\sigma^2}}}\leq \exp(1),\quad \forall \ell=0,\cdots,H-1,
\end{equation}
where recall that $\E_{kH}[\cdot]:=\E[\cdot\mid\cF_{kH-1}]$.
We have thus proved that $\{n_{k,\ell}\}_{\ell=0}^{H-1}$ satisfies the conditions in Lemma~\ref{lem:VR_subG}. Applying Lemma~\ref{lem:VR_subG} then yields: for any $\delta\in(0,1)$,
\begin{equation}
\medmath{
\P_{kH}\paren{\norm{\frac{1}{H}\sum_{\ell=0}^{H-1}n_{k,\ell}}\leq c\sqrt{\frac{\sigma^2}{H}\log\frac{2d}{\delta}}}\geq 1-\delta.
}
\end{equation}
It then follows from using the tower property (similarly as in~\eqref{eqn:subg_proof}) that with probability at least $1-\delta$,
\begin{equation}
\medmath{
    \norm{n_k}\leq c\sqrt{\frac{\sigma^2}{H}\log\frac{2d}{\delta}}.} 
\end{equation}
Union bounding over all $K\leq T$ batches and squaring both sides then yields the desired claim.
\end{proof}

We can now complete the proof of Lemma~\ref{lem:onestep_sc} by first establishing a one-step descent (similar to Lemma~\ref{lem:onestep}), and then controlling the resulting cross term via Young's inequality coupled with Lemma~\ref{lem:VR_sn}.
Using the smoothness of $f$ and selecting $\alpha\leq 1/(2L)$, we obtain:
\begin{equation}
\resizebox{1\hsize}{!}{$
    r^{(k+1)}\leq r^{(k)} - \frac{\alpha}{2}\norm{\nabla f(x^{(k)})}^2-\alpha\iprod{\nabla f(x^{(k)})}{n_k}+L\alpha^2\norm{n_k}^2.\label{eqn:onestep_batch}
$}
\nonumber
\end{equation}
Using Young's inequality to bound the cross term, we have
\begin{equation}
\medmath{
\begin{aligned}
    r^{(k+1)}&\leq r^{(k)}- \frac{\alpha}{2}\norm{\nabla f(x^{(k)})}^2  +L\alpha^2\norm{n_k}^2\\
    &\ +\frac{\alpha}{4}\norm{\nabla f(x^{(k)})}^2 +\alpha\norm{n_k}^2\\
    &\leq r^{(k)}- \frac{\alpha}{4}\norm{\nabla f(x^{(k)})}^2  +2\alpha\norm{n_k}^2,\label{eqn:thm_batch_int}
\end{aligned}
}
\end{equation}
where we use the fact that $\alpha L\leq 1$, implying  $L\alpha^2\leq \alpha$. Now recall the 
gradient domination property of strongly convex functions: $\Vert \nabla f(x) \Vert^2_2 \geq 2 \mu (f(x) - f(x^*)), \forall x \in \mathbb{R}^d$. Using this and plugging in~\eqref{eqn:VR_noise} from Lemma~\ref{lem:VR_sn} into~\eqref{eqn:thm_batch_int} yields with probability at least $1-\delta$, for all $k\geq 0$:
\begin{equation}
\begin{aligned}
    r^{(k+1)} &\leq \paren{1-\frac{\alpha\mu}{2}}r^{(k)}+\frac{2c^2\sigma^2\alpha}{H}\log\frac{2dT}{\delta}.
\end{aligned}
\end{equation}
Iterating from $k=0$ to $K-1$ leads to the claim of Lemma~\ref{lem:onestep_sc}.

We now show that our proof technique extends readily to non-convex functions as well. 
\begin{theorem}[Non-Convex Objective, \texttt{Batched SGD}]\label{thm:nc}
Suppose Assumptions~\ref{ass:smoothness},~\ref{ass:unbiased},~\ref{ass:subG} hold. Then, given any $\delta\in(0,1)$, there exists a universal constant $c$ such that with $K=\sqrt{T}$ and $\alpha\leq1/(2L)$, the following holds for \texttt{Batched SGD} w.p. at least $1-\delta$: 
\begin{equation}
\frac{1}{K}\sum_{k=0}^{K-1}\norm{\nabla f(x^{(k)})}^2\leq \frac{8Lr^{(0)}}{\sqrt{T}}+\frac{8c^2\sigma^2}{\sqrt{T}}\log\frac{2dT}{\delta}.\label{eqn:nonconvex_sgd}
\end{equation}
\end{theorem}
\begin{proof}
First, note that arriving at~\eqref{eqn:thm_batch_int} only requires smoothness of $f$. As such, telescoping~\eqref{eqn:thm_batch_int} from $k=0$ to $K-1$ and dividing both sides by $\alpha K/4$ yields:
\begin{equation}
\frac{1}{K}\sum_{k=0}^{K-1}\norm{\nabla f(x^{(k)})}^2\leq \frac{4\paren{r^{(0)}-r^{(K)}}}{\alpha K}+\frac{8}{K}\sum_{k=0}^{K-1}\norm{n_k}^2.\label{eqn:nc_int}
\end{equation}
Removing the negative term on the R.H.S., selecting $\alpha=1/(2L)$, and plugging in~\eqref{eqn:VR_noise} from Lemma~\ref{lem:VR_sn} into~\eqref{eqn:nc_int} yields w.p. at least $1-\delta$,
\begin{equation}
\frac{1}{K}\sum_{k=0}^{K-1}\norm{\nabla f(x^{(k)})}^2\leq \frac{8Lr^{(0)}}{K}+\frac{8c^2\sigma^2}{H}\log\frac{2dT}{\delta}.
\end{equation}
Selecting $K=\sqrt{T}$ then establishes the desired claim.
\end{proof}
\textbf{Main Takeaways.} Theorem~\ref{thm:nc} shows that the guarantee obtained by \texttt{Batched SGD} also extends to the general non-convex case, achieving a near-optimal rate of $\bigo{\log(dT/\delta)/\sqrt{T}}$. Compared with prior works~\cite{Orabona,liu2023revisiting}, this rate incurs only an extra $\log d$ factor as in the strongly convex case, while the proof remains significantly simpler. 
\begin{remark} Interestingly, to achieve near-optimal rates, while $K=\mc{O}(\log(T))$ updates sufficed for the strongly-convex case, the non-convex case requires $K=\sqrt{T}$ updates (from Theorem~\ref{thm:nc}). Thus, favorable structural properties of $f$ allow for less frequent parameter updates. 
\end{remark}

\section{Extension to The Federated Setting}\label{sec:FL}
In this section, we extend the batching principle developed in Section~\ref{sec:bsgd} to the federated setting. Recall that the \textbf{core idea} of \texttt{Batched SGD} is to refine the stochastic gradient using multiple samples before performing an update, thereby reducing the variance of the noise, simplifying the high-probability analysis, and saving the number of update rounds.

Interestingly, the same idea naturally translates to the federated setting, where local gradient computations can be viewed as forming batches between successive communication rounds. From this perspective, \textit{the batch-size $H$ directly informs the communication frequency between the central server and the agents}. We show that this batching idea leads to logarithmic rounds of communication and a simple high-probability analysis for federated optimization.

\subsection{Problem Formulation}
We consider a setting involving a central server and $M$ agents, where each agent $i\in[M]$ is associated with a local objective function $f_i:\R^d\to\R$. The  standard federated learning~\cite{mcmahan,konevcny} problem seeks to solve: 
\begin{equation}
    \min_{x\in\R^d} f(x):=\frac{1}{M}\sum_{i\in[M]}f_i(x),\label{eqn:global_f}
\end{equation}
subject to stringent privacy and communication constraints. To capture the standard ``intermittent communication model" in FL~\cite{mcmahan, konevcny}, suppose each agent $i \in [M]$ has access to $T$ online samples, and let $\xi_{i,t}$ denote the noise sample for agent $i$ at time-step $t$. These $T$ samples are partitioned into $K$ communication rounds of duration $H$ each, such that $T=KH$. The standard FL algorithm proposed in~\cite{mcmahan} proceeds as follows. At the beginning of each communication round $k=0,\cdots,K-1$, the server broadcasts the current global parameter $\bar x^{(k)}$ to all agents. Each agent then initializes its local parameter from 
$\bar x^{(k)}$, and performs $H$ local parameter updates (in isolation) of the following form: 
\begin{equation}
    x_{i,\ell+1}^{(k)} = x_{i,\ell}^{(k)} - \alpha g_i(x_{i,\ell}^{(k)}, \xi_{i,kH+\ell}), \ell=0,\cdots, H-1. \label{eqn:fl_update}
\end{equation}
Here, $\alpha$ is the step-size, $x_{i,\ell}^{(k)}$ denotes the local parameter of agent $i$ at local iteration $\ell$ of round $k$, $g_i(x_{i,\ell}^{(k)}, \xi_{i,kH+\ell})$ is a noisy version of $\nabla f_i(x_{i,\ell}^{(k)})$, and $\xi_{i,kH+\ell}$ is the RV capturing the randomness in $g_i$. After $H$ local updates, the server aggregates the local models to update the global parameter:
\begin{equation}
    \bar x^{(k+1)} = \frac{1}{M}\sum_{i\in[M]}x_{i,H}^{(k)}.\label{eqn:aggr_naive}
\end{equation}

The new global model $\bar x^{(k+1)}$ is then broadcast to all agents to initiate the next round. Although the \emph{expected} behavior of the above algorithm (and its numerous variants) has been extensively analyzed~\cite{mitra2021linear,karimireddy2020scaffold}, we are unaware of any non-trivial high-probability bounds in FL. To close this gap, we start by making the following assumptions, which are natural FL counterparts of those in Section~\ref{sec:prob_form}. 

\begin{assumption}[Unbiased Gradient]\label{ass:unbiased_fed}
For each agent $i\in[M]$, the stochastic gradient $g_i(x, \xi)$ is unbiased, i.e., $\E_\xi [g_i(x,\xi)] = \nabla f_i(x), \forall x\in\R^d.$
\end{assumption}
\begin{assumption}[Norm-Sub-Gaussian Noise]\label{ass:norm_subG_fed}
For each agent $i\in[M]$, the noise in $g_i(x, \xi)$ is norm-sub-Gaussian, i.e., 
\begin{equation}
    \E_\xi\bracket{\exp\paren{\norm{g_i(x,\xi)-\nabla f_i(x)}^2/\sigma^2}}\leq \exp\paren{1}, \forall x \in \mathbb{R}^d, 
    \label{eqn:subG_fed}
\end{equation}
where $\sigma^2 > 0$ is the variance proxy.
\end{assumption}
\begin{assumption}[Independence]\label{ass:indep}
The stochastic processes $\{\xi_{i,t}\}$ and $\{\xi_{j,t}\}$ are statistically independent $\forall i\neq j\in[M]$. 
\end{assumption}
Assumption~\ref{ass:indep} is standard, and needed to obtain collaborative speedups w.r.t. the number of agents; see~\cite{karimireddy2020scaffold, zhuachieving, mangold2024scafflsa}. 

\subsection{Batched SGD for FL}
Similar to the SGD update rule in~\eqref{eqn:SGD}, in standard FL algorithms like \texttt{FedAvg}~\cite{mcmahan}, the local parameter $x_{i,\ell}^{(k)}$ of each agent $i\in[M]$ is updated $H$ times in each round via~\eqref{eqn:fl_update}; each such update is made using a noisy stochastic gradient constructed from a \textbf{single} sample $\xi_{i,kH+\ell}$. Unlike in the single-agent case, this vanilla update rule results in what is known as the ``client-drift'' issue in the FL setting~\cite{mitra2021linear, karimireddy2020scaffold}. Specifically, since each agent $i\in[M]$ updates its local model $H$ times during each communication round without synchronization, the local model of agent $i$ tends to drift toward its own local optimum $x_i^* \in \argmin_x f_i(x)$. Consequently, naively aggregating the local models as per~\eqref{eqn:aggr_naive} yields a \textit{bias term} in the final guarantee that hinders the benefits of collaboration~\cite{mitra2021linear, karimireddy2020scaffold}. 

In what follows, we will show that this drift issue can be completely resolved by using the batching idea. Concretely, during each round $k$, instead of making multiple local updates using highly noisy stochastic gradients, we advocate making \textbf{one single update using a refined gradient} $g_{i,k}$ constructed using the samples collected during round $k$:  
\begin{equation}
    g_{i,k}:=\frac{1}{H}\sum_{\ell=0}^{H-1}g_i(\bar x^{(k)}, \xi_{i,kH+\ell}).\label{eqn:refined_gradient_fed}
\end{equation}
Agent $i$ then makes a \emph{single} update to its local model as: 
\begin{equation}
    x_i^{(k)} = \bar x^{(k)} - \alpha g_{i,k},\label{eqn:update-fl}
\end{equation}
where $x_i^{(k)}$ denotes the local model of agent $i$ at the end of round $k$. Finally, the server aggregates the local models across agents and updates the global parameter as follows:
\begin{equation}
    \bar x^{(k+1)} = \frac{1}{M}\sum_{i\in[M]}x_i^{(k)}.\label{eqn:aggr_fl}
\end{equation}
We call the above method \texttt{Batched FL}, and summarize it in Algorithm~\ref{algo:batched-fl}, with an illustration in Fig.~\ref{fig:fl}.

\begin{figure}
    \centering
    \includegraphics[scale = 0.5]{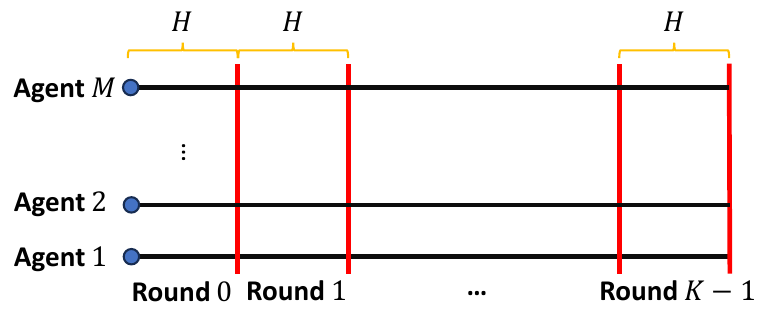}
    \caption{Illustration of \texttt{Batched FL}. Each agent in $[M]$ partitions its $T$ samples into $K$ communication rounds. At the red bars, agents update their local parameters and the server aggregates them.}
    \label{fig:fl}
\end{figure}
\begin{algorithm}[t!]
\caption{\texttt{Batched FL}} 
\label{algo:batched-fl}
\begin{algorithmic}[1]
\State \textbf{Input:} Step-size $\alpha$, initial parameter $\bar x^{(0)}$, number of samples $T$, number of communication rounds $K$, number of samples in each round $H$, number of agents $M$.
\For {$k=0,\ldots,K-1$} 
\For {$i=1,\ldots, M$}
\For {$\ell=0,\ldots, H-1$}
\State Agent $i$ computes $g_i(\bar x^{(k)}, \xi_{i,kH+\ell})$
\EndFor
\State Agent $i$ constructs $g_{i,k}$ as per~\eqref{eqn:refined_gradient_fed}
\State Agent $i$ computes local model $x_i^{(k)}$ via~\eqref{eqn:update-fl} and transmits it to the server
\EndFor
\State Server broadcasts $\bar x^{(k+1)}$ computed as in~\eqref{eqn:aggr_fl}
\EndFor
\end{algorithmic}
\end{algorithm}
\begin{remark} Departing fundamentally from the standard FL template~\cite{mcmahan, konevcny, mitra2021linear, karimireddy2020scaffold, zhuachieving, mangold2024scafflsa} where multiple local updates are made per round using noisy gradients, the essence of \texttt{Batched FL} lies in making a single precise update using a direction that enjoys an $H$-fold lower variance. This ensures that the local models are always \textbf{synchronized}, thereby preventing biases from heterogeneity in local loss functions, while preserving collaborative benefits. We later show that this mechanism only requires \textbf{logarithmically} many communication rounds for strongly convex objectives.
\end{remark}
\subsection{Results And Analysis}
In this section, we present the high-probability rates for \texttt{Batched FL}. To do so, let us define $r^{(k)}:=f(\bar x^{(k)})-f(x^*)$ as the global function sub-optimality gap in round $k$, where $x^*:=\arg\min_x f(x)$.  We then have the following theorem for strongly convex loss functions. 
\begin{theorem}[Strongly Convex Objective, \texttt{Batched FL}]\label{thm:fl_sc}
Suppose the global objective function $f$ in~\eqref{eqn:global_f} satisfies Assumptions~\ref{ass:smoothness},~\ref{ass:SC}, and suppose Assumptions~\ref{ass:unbiased_fed},~\ref{ass:norm_subG_fed},~\ref{ass:indep} hold. Then, given any $\delta\in(0,1)$, there exists a universal constant $c$ such that with $K=4\kappa\log MT$ and $\alpha=1/(2L)$, the following holds for \texttt{Batched FL} w.p. at least $1-\delta$: 
\begin{equation}
    r^{(K)}\leq \frac{r^{(0)}}{MT} + \frac{16c^2\kappa\sigma^2}{\mu MT}\log MT\log\frac{2dT}{\delta}.\label{eqn:sc_fl}
\end{equation}
\end{theorem}
\begin{proof}
Similar to Lemma~\ref{lem:onestep}, using smoothness of $f$, defining $g_k:=(1/M)\sum_{i\in[M]}g_{i,k}$, $n_k:=g_k - \nabla f(\bar x^{(k)})$ and selecting $\alpha L\leq 1/2$, we have the one-step descent
\begin{equation}
\begin{aligned}
    r^{(k+1)}&\leq r^{(k)}-\frac{\alpha}{4}\norm{\nabla f(\bar x^{(k)})}^2+2\alpha\norm{n_k}^2.\label{eqn:one-step-fl}
\end{aligned}
\end{equation}

Next, we need to bound the squared-norm term $\norm{n_k}^2$, which follows a similar process as in the proof of Lemma~\ref{lem:VR_sn}. Fixing a round $k$, as per the definition of $n_k$, we have 
\begin{equation}
\medmath{
\begin{aligned}
    n_k&=\frac{1}{M}\sum_{i\in[M]}\paren{\frac{1}{H}\sum_{\ell=0}^{H-1}g_i(\bar x^{(k)}, \xi_{i,kH+\ell}) - \nabla f_i(\bar x^{(k)})}\\
    &=\frac{1}{MH}\sum_{i\in[M]}\sum_{\ell=0}^{H-1}n_{i,k,\ell},
\end{aligned}
}
\end{equation}
where we define $n_{i,k,\ell}:=g_i(\bar x^{(k)}, \xi_{i,kH+\ell}) - \nabla f_i(\bar x^{(k)})$. We now verify that $\{n_{i,k,\ell}\}_{i,\ell}$ satisfies the conditions for invoking the variance-reduction effect in  Lemma~\ref{lem:VR_subG}. To do so, let $\cF_{kH-1}:=\{\xi_{i,t}\}_{i\in[M], t\leq kH-1}$ be the sigma-algebra generated by all the randomness across all agents in $[M]$ up to time-step $kH-1$ ($\cF_{-1}=\varnothing$). Observe that $\bar x^{(k)}$ is $\cF_{kH-1}$-measurable. Conditioned on $\cF_{kH-1}$ then, from Assumption~\ref{ass:indep},  $\{n_{i,k,\ell}\}_{i,\ell}$ is \textbf{independent} since the only randomness lies in $\{\xi_{i,kH+\ell}\}_{i,\ell}$; from Assumption~\ref{ass:unbiased_fed}, it is \textbf{zero-mean} due to the unbiasedness of $g_i$'s. Finally, Assumption~\ref{ass:norm_subG_fed} states that each $n_{i,k,\ell}$ is $\sigma^2$\textbf{-norm-sub-Gaussian}, i.e.,
\begin{equation}
    \E\bracket{\exp\paren{\frac{\norm{n_{i,k,\ell}}^2}{\sigma^2}}\mid \cF_{kH-1}}\leq \exp(1),\quad \forall i,\ell.
\end{equation}
Now that we have shown that the noise sequence $\{n_{i,k,\ell}\}_{i,\ell}$ satisfies all the requirements for Lemma~\ref{lem:VR_subG}, invoking it and using the tower property to get rid of the conditioning on $\cF_{kH-1}$ yields for any $\delta\in(0,1)$, w.p. at least $1-\delta$:
\begin{equation}
    \norm{n_k}\leq c\sqrt{\frac{\sigma^2}{MH}\log\frac{2d}{\delta}}.
\end{equation}
Union bounding over all $K\leq T$ rounds and squaring both sides yields w.p. at least $1-\delta$,
\begin{equation}
    \norm{n_k}^2\leq \frac{c^2\sigma^2}{MH}\log\frac{2dT}{\delta}.\label{eqn:sn_fl}
\end{equation}
Plugging this into~\eqref{eqn:one-step-fl} and using the gradient domination property of $f$ yields w.p. at least $1-\delta$,
\begin{equation}
\medmath{
}    r^{(k+1)}\leq \paren{1-\frac{\alpha\mu}{2}}r^{(k)}+\frac{2c^2\sigma^2\alpha}{MH}\log\frac{2dT}{\delta}.
\end{equation}
Iterating from $k=0$ to $K-1$ and selecting $\alpha=1/(2L)$ yields w.p. at least $1-\delta$,
\begin{equation}
\medmath{
    r^{(K)}\leq \exp\paren{-\frac{K}{4\kappa}}r^{(0)}+\frac{4c^2\sigma^2}{\mu MH}\log\frac{2dT}{\delta}.}
\end{equation}
Selecting $K=4\kappa\log MT$ yields the desired claim.
\end{proof}
A few comments are in order regarding Theorem~\ref{thm:fl_sc}.

$\bullet$ \textit{High-Probability Guarantees for FL.} The significance of Theorem~\ref{thm:fl_sc} lies in providing the first near-optimal high-probability rate on the order of $\mc{O}\left({\log(MT)\log\paren{dT/\delta}/(MT)}\right)$ in FL. We note also that Theorem~\ref{thm:fl_sc} applies to a fairly general setting, where the local loss functions can be arbitrarily heterogeneous, and they need not be smooth or even convex. 

$\bullet$ \textit{Linear Speedup with No Heterogeneity Bias.} Comparing the rate in~\eqref{eqn:final_sc} with that in~\eqref{eqn:sc_fl}, we note that the latter is tighter by a factor of $M$. This improvement by a factor of $M$ is typically referred to as a ``linear-speedup" effect, and stems from collaboration. Notably, we achieve such an effect without relying on 
sophisticated control-variate/gradient-tracking techniques used in prior FL work~\cite{karimireddy2020scaffold, mitra2021linear, zhuachieving, mangold2024scafflsa}.  

$\bullet$ \textit{Logarithmic Communication Rounds.} The batching idea informs the communication frequency in the federated setting just as it informs the number of updates in the single-agent setting. Specifically, since $K=\bigo{\log MT}$, only logarithmic communication rounds between the server and the agents suffice to achieve near-optimal rates, significantly reducing the communication overhead.

Similarly to Section~\ref{sec:analysis}, we can develop high-probability guarantees for a non-convex $f$ as well. 

\begin{theorem}[Non-Convex Objective, \texttt{Batched FL}]\label{thm:nc_fl}
Suppose the global objective function $f$ in~\eqref{eqn:global_f} satisfies Assumption~\ref{ass:smoothness}, and suppose Assumptions~\ref{ass:unbiased_fed},~\ref{ass:norm_subG_fed},~\ref{ass:indep} hold. Then, given any $\delta\in(0,1)$, there exists a universal constant $c$ such that with $K=\sqrt{MT}$ and $\alpha=1/(2L)$, the following holds for \texttt{Batched FL} w.p. at least $1-\delta$:
\begin{equation}
\frac{1}{K}\sum_{k=0}^{K-1}\norm{\nabla f(\bar x^{(k)})}^2\leq \frac{8Lr^{(0)}}{\sqrt{MT}}+\frac{8c^2\sigma^2}{\sqrt{MT}}\log\frac{2dT}{\delta}.
\end{equation}
\end{theorem}
\begin{proof}
The proof follows the exact same telescoping argument used for proving Theorem~\ref{thm:nc}, where we now use the bound on $\Vert n_k \Vert^2_2 $ in~\eqref{eqn:sn_fl}, and set $K=\sqrt{MT}.$ 
\end{proof}
\textbf{Main Takeaways.} Theorem~\ref{thm:nc_fl} establishes a high-probability guarantee for general non-convex objectives on the order of $\bigo{\log(dT/\delta)/\sqrt{MT}}$, once again demonstrating an explicit linear speedup with respect to the number of agents $M$. Compared to the strongly-convex setting, however, a higher communication frequency of $K= \sqrt{MT}$ is needed to achieve this speedup. 

\section{Conclusion}
We studied the problem of deriving high-probability bounds in the context of stochastic optimization. To that end, we introduced a simple variant of SGD called \texttt{Batched SGD} which partitions the online data stream into multiple batches, constructs low variance gradient directions per batch, and makes parameter updates using such refined directions. For this variant, we established near-optimal high-probability concentration bounds for smooth strongly-convex and non-convex functions. Notably, our proofs are significantly simpler relative to prior work and require no restrictive assumptions. Finally, we showed how our batching idea can be naturally extended to the federated learning setting. We conjecture that the basic idea employed in this paper can be used to generate simple concentration proofs for other stochastic approximation problems as well, including those that arise in reinforcement learning. Formalizing this conjecture is part of our ongoing work. 

\bibliographystyle{IEEEtran} 
\bibliography{bib}
\end{document}